\documentclass[11pt]{article}

\usepackage[letterpaper,left=1.12in,right=1.12in,top=1.00in,bottom=1.05in]{geometry}
\usepackage{amsmath,amssymb,amsthm,mathtools}
\usepackage{microtype}
\usepackage{enumitem}
\usepackage[hidelinks]{hyperref}
\usepackage{cite}

\hypersetup{
	colorlinks=true,
	linkcolor=blue,
	filecolor=blue,
	urlcolor=blue,
	citecolor=cyan,
}

\newtheorem{theorem}{Theorem}[section]
\newtheorem{lemma}{Lemma}[section]

\newtheorem{claim}{Claim}[section]

\newtheorem{conjecture}{Conjecture}[section]

\allowdisplaybreaks
\numberwithin{equation}{section}

\newcommand{\calA}{\mathcal A}
\newcommand{\calB}{\mathcal B}
\newcommand{\calF}{\mathcal F}
\newcommand{\calG}{\mathcal G}
\newcommand{\calS}{\mathcal S}

\newcommand{\calW}{\mathcal W}

\begin{document}

\title{Intersecting families of sets are usually trivial for $n\ge 2k+3$}

\author{Jiabao Yang\footnote{Email: jbyang1215@nju.edu.cn}\\
\small{School of Mathematics, Nanjing University, Nanjing 210093, P.R. China}}

\date{}

\maketitle

\vspace{1.1em}
\begin{abstract}
A family of subsets of $[n]$ is called intersecting if it contains no pair of disjoint sets. 
It is called trivial if all its members contain a common element. 
Frankl and Kupavskii, and independently Balogh, Das, Liu, Sharifzadeh, and Tran, proved that there is a constant $c>0$ such that, 
whenever $n \geq 2k+2+c\sqrt{k\ln k}$, almost all $k$-uniform intersecting families are trivial. 
Balogh, Garcia, Li, and Wagner later improved this range to $n \geq 2k+100\ln k$. 
In this paper, we prove that the same conclusion holds for every $n\geq 2k+3$.
This verifies the conjectured conclusion of Balogh, Garcia, Li, and Wagner throughout this range.
\end{abstract}

\section{Introduction}

For a positive integer $n$, we write $[n]=\{1,\ldots,n\}$.
A family $\mathcal{F}\subseteq 2^{[n]}$ is called \emph{intersecting} if any two sets in $\mathcal{F}$ have a nonempty intersection. 
The classical theorem of Erd\H{o}s, Ko, and Rado~\cite{ErdosKR1961} states that if $\mathcal{F}\subseteq \binom{[n]}{k}$ is intersecting and $n\geq 2k$,
then $|\mathcal{F}|\leq \binom{n-1}{k-1}$.
A full star is the family of all $k$-subsets of $[n]$ that contain a fixed element. Since a full star has exactly
$\binom{n-1}{k-1}$ members, the bound above is best possible. More
generally, an intersecting family is called \emph{trivial} if all its
members contain a common element.
The classical Hilton--Milner theorem~\cite{HiltonMilner1967} gives the sharp upper bound for non-trivial intersecting families. If $k\geq3$, $n>2k$, and $\mathcal{F}\subseteq\binom{[n]}k$ is a non-trivial intersecting family, then
$
|\mathcal{F}|\leq\binom{n-1}{k-1}-\binom{n-k-1}{k-1}+1.
$
Thus every non-trivial intersecting family is smaller than a full star by at least $\binom{n-k-1}{k-1}-1$.
For a broad overview of intersection problems for finite sets, see Frankl and Tokushige~\cite{FranklTokushige2016}.

Let $I(n,k)$ denote the number of $k$-uniform intersecting families on $[n]$, 
and let $I(n,k,\geq 1)$ denote the number of such families that are non-trivial.
In 2015, Balogh, Das, Delcourt, Liu, and Sharifzadeh~\cite{BaloghDasDelcourtLiuSharifzadeh2015} proved the following theorem.

\begin{theorem}[Balogh, Das, Delcourt, Liu, and Sharifzadeh~\cite{BaloghDasDelcourtLiuSharifzadeh2015}]\label{thm:BaloghDasDelcourtLiuSharifzadeh2015}
For $n\ge 2k+1$ and $k\to\infty$, we have
$$I(n,k)=2^{(1+o(1))\binom{n-1}{k-1}}.$$
\end{theorem}

Improving a result of Balogh, Das, Delcourt, Liu, and Sharifzadeh~\cite{BaloghDasDelcourtLiuSharifzadeh2015}, 
Frankl and Kupavskii~\cite{FranklKupavskii2018}, and independently Balogh, Das, Liu, Sharifzadeh, and Tran~\cite{BaloghDasLiuSharifzadehTran2019},
proved that there is a constant $c>0$ such that, if $n\geq 2k+2+c\sqrt{k\ln k}$, 
then almost all $k$-uniform intersecting families are trivial. 
More precisely, Frankl and Kupavskii~\cite{FranklKupavskii2018} proved the result with $c=2$, while Balogh, Das, Liu, Sharifzadeh, and Tran~\cite{BaloghDasLiuSharifzadehTran2019} proved it for a sufficiently large absolute constant $c$.

\begin{theorem}[Frankl and Kupavskii~\cite{FranklKupavskii2018}, and Balogh, Das, Liu, Sharifzadeh, and Tran~\cite{BaloghDasLiuSharifzadehTran2019}]\label{thm:FranklKupavskii2018}
For $n\ge 2k+2+c\sqrt{k\ln k}$ and $k\to\infty$, we have
$$I(n,k)=(n+o(1))2^{\binom{n-1}{k-1}} \quad \text{and} \quad I(n,k,\geq 1)=o\left(2^{\binom{n-1}{k-1}}\right),$$
where $c =2$ in~\cite{FranklKupavskii2018} and $c$ was a large constant in~\cite{BaloghDasLiuSharifzadehTran2019}.
\end{theorem}

Later, Balogh, Garcia, Li, and Wagner~\cite{BaloghGarciaLiWagner2024} extended the range in Theorem~\ref{thm:FranklKupavskii2018} as follows.

\begin{theorem}[Balogh, Garcia, Li, and Wagner~\cite{BaloghGarciaLiWagner2024}]\label{thm:BaloghGarciaLiWagner2024}
For $n\ge 2k+100\ln k$ and $k\to\infty$, we have
$$I(n,k)=(n+o(1))2^{\binom{n-1}{k-1}} \quad \text{and} \quad I(n,k,\geq 1)=o\left(2^{\binom{n-1}{k-1}}\right).$$
\end{theorem}

They also proposed the following conjecture.
\begin{conjecture}[Balogh, Garcia, Li, and Wagner~\cite{BaloghGarciaLiWagner2024}]\label{conj:BaloghGarciaLiWagner2024}
For $n\ge 2k+2$ and $k\to\infty$, we have
$$I(n,k,\geq 1)=o(I(n,k)).$$
\end{conjecture}

The two smaller cases exhibit different behavior.
When $n=2k$, two members of $\binom{[n]}k$ are disjoint if and only if they are complements, so an intersecting family contains at most one member from each complementary pair. 
When $n=2k+1$, Balogh, Garcia, Li, and Wagner~\cite{BaloghGarciaLiWagner2024} showed that trivial families do not form almost all intersecting families and proposed a different description of the typical structure. Thus $n=2k+2$ is the first value at which typical triviality can be expected.

In this article, our goal is to determine how far the lower bound on $n$ can be reduced by combining a global Kneser-graph container with their local inclusion-graph method. 
The following theorem states that their conjecture holds except for the boundary case $n=2k+2$.

\begin{theorem}\label{thm:main}
For $n\ge 2k+3$ and $k\to\infty$, the number of non-trivial intersecting families in $\binom{[n]}k$ is $o\left(2^{\binom{n-1}{k-1}}\right)$.
In particular,
$$I(n,k)=(n+o(1))2^{\binom{n-1}{k-1}} \quad \text{and} \quad I(n,k,\geq 1)=o\left(2^{\binom{n-1}{k-1}}\right).$$
Moreover, this result confirms Conjecture~\ref{conj:BaloghGarciaLiWagner2024} for $n\geq 2k+3$.
\end{theorem}

Our proof proceeds in two stages. 
We first use a global container theorem for the Kneser graph to reduce the problem to intersecting families lying close to a full star, 
since all remaining containers already have a sufficient deficit in the exponent. 
After fixing such a star, we encode the members outside the star and the excluded members inside the star by an independent set in an appropriate inclusion graph. 
We then divide the argument according to the size of the closure of the non-star part. 
Families with large closure are controlled by a direct entropy estimate, 
while in the small closure case we decompose the closure into $2$-linked components and count each component through a “good pair and reconstruction” procedure. 
The main new ingredient is the reconstruction estimate, which yields a factor of the form $2^{-g/(k\sqrt{d_X})}$ for a component with neighborhood size $g$ and is strong enough to complete the summation over all components.

\section{Preliminaries}\label{sec:Preliminary}

\subsection{Notation}

Let $G=(V,E)$ be a simple graph. For $A\subseteq V$, the neighborhood of $A$ in $G$, denoted by $N_G(A)$, is the set of vertices adjacent to at least one vertex of $A$. 
If $u$ and $v$ lie in different components, we set $d_G(u,v)=\infty$.
For $S\subseteq V$, the neighborhood of $A$ restricted to $S$ is $N_{S,G}(A)=N_G(A)\cap S$. The distance $d_G(u,v)$ between two vertices $u,v\in V$ is the length of a shortest path from $u$ to $v$ in $G$. 
We write $e_G(A)$ for the number of edges in the subgraph induced by $A$. 
A set $I\subseteq V$ is \emph{independent} if no two vertices of $I$ are adjacent. 
When the underlying graph is clear, we write $N(A)$, $N_S(A)$, $d(u,v)$, and $e(A)$.

Let $\Sigma$ be a bipartite graph with classes $X$ and $Y$. For $A \subseteq X$, define the \emph{closure} of $A$ by $$
[A] := \{ v \in X : N(v) \subseteq N(A) \}
$$
and call $A$ \emph{closed} if $[A] = A$. 
We say that a pair of vertices $u$ and $v$ is $k$-linked in a set $A \subseteq X$ if there exists a sequence $u = v_1, v_2, \dots, v_{\ell-1}, v_\ell = v$ in $A$ such that for each $i \in [\ell-1]$, the distance $d_\Sigma(v_i, v_{i+1})$ is at most $k$. A set $A$ is $k$-linked if every pair of vertices in $A$ is $k$-linked in $A$. A $k$-linked component of a set $B \subseteq X$ is an inclusion-maximal $k$-linked subset of $B$.

Given integers $1 \le k \le n/2$, the Kneser graph $KG(n,k)$ is defined on the vertex set 
$V = \binom{[n]}{k}$, where two $k$-sets $F_1,F_2 \in \binom{[n]}{k}$ are adjacent if and only if $F_1 \cap F_2 = \varnothing$.

For $x\in[n]$, define the full star centered at $x$ by $\calS_x=\{F\in\tbinom{[n]}k:x\in F\}$.

For two sets $A$ and $B$, their symmetric difference is $A\mathbin\triangle B=(A\setminus B)\cup(B\setminus A)$. 
Throughout the paper, $\log$ denotes the logarithm to base $2$, while $\ln$ denotes the natural logarithm. 
For real $z,u\ge0$, we write $\binom{z}{\le u}=\sum_{j=0}^{\lfloor u\rfloor}\binom{\lceil z\rceil}{j}$.

\subsection{Tools}

Das and Tran~\cite{DT} proved a removal lemma that holds whenever the size of $\mathcal{F}$ is close to the size of a union of $\ell$ full stars and has relatively few disjoint pairs.
For earlier and related stability results for intersecting families, see Friedgut~\cite{Friedgut2008}, Keevash~\cite{Keevash2008}, and Ellis, Keller, and Lifshitz~\cite{EllisKellerLifshitz2019}.
The following convenient consequence is obtained from Theorem 1.2 of Das and Tran~\cite{DT} by taking $\ell=1$.

\begin{theorem}[Das and Tran~\cite{DT}]\label{thm:DT}
There is an absolute constant $C_{\rm DT}>1$ with the following property. 
Let $n>2k$, let $\alpha\in\mathbb R$ and $\beta\ge0$, and suppose that $\calF\subseteq\binom{[n]}k$ has size
\begin{equation*}
|\calF|=(1-\alpha)\binom{n-1}{k-1}
\end{equation*}
and contains at most $\beta\binom{n-1}{k-1}\binom{n-k-1}{k-1}$ disjoint pairs. If
\begin{equation*}
\max\{2|\alpha|,\beta\}\le\frac{n-2k}{(20C_{\rm DT})^2n},
\end{equation*}
then some star $\calS_x$ satisfies
\begin{equation*}
|\calF\mathbin\triangle\calS_x|\le C_{\rm DT}(|\alpha|+2\beta)\frac{n}{n-2k}\binom{n-1}{k-1}.
\end{equation*}
\end{theorem}

The quantity $\min_{x\in[n]}|\mathcal{F}\setminus\mathcal{S}_x|$, which measures the one-sided distance from $\mathcal{F}$ to the nearest full star, is often called the \emph{diversity} of $\mathcal{F}$. Several extremal results on this parameter were obtained by Kupavskii~\cite{Kupavskii2018} and by Frankl and Wang~\cite{FranklWang2024}.

We next recall the connection between the Kneser graph and the Erd\H{o}s--Ko--Rado theorem.
Since edges of $KG(n,k)$ denote disjoint pairs in $\binom{[n]}{k}$, it follows that independent sets of $KG(n,k)$ correspond directly to intersecting families in $\binom{[n]}{k}$. 
Thus the Erd\H{o}s--Ko--Rado theorem, viewed from the perspective of the Kneser graph, shows that the independence number of 
$KG(n,k)$ is $\binom{n-1}{k-1}$ when $n \ge 2k$.

The following theorem gives the spectrum of the Kneser graph. 
This formulation is also recorded by Das and Tran~\cite[Section~2]{DT}. 
For a systematic treatment of the Kneser graph, the ratio bound, and the Johnson scheme, see Godsil and Meagher~\cite[Chapters~2 and~6]{GodsilMeagher2015}.

\begin{theorem}[Lov\'asz~\cite{LovaszShannon}]\label{thm:spectrum}
The eigenvalues of the adjacency matrix of $KG(n,k)$ are
\begin{equation*}
(-1)^j\binom{n-k-j}{k-j},\quad 0\le j\le k.
\end{equation*}
In particular, its largest eigenvalue is the common degree of the vertices 
in the regular graph $KG(n,k)$, which is $\binom{n-k}{k}$, and its least eigenvalue is $-\binom{n-k-1}{k-1}$. 
Moreover, every $W\subseteq V(KG(n,k))$ with $|W|\ge\binom{n-1}{k-1}$ satisfies
\begin{equation*}
2e(W)\ge\frac{\binom{n-k}{k}+\binom{n-k-1}{k-1}}{\binom nk}|W|\left(|W|-\binom{n-1}{k-1}\right).
\end{equation*}
\end{theorem}

For the next result, let $n=2k+q$, fix $x\in[n]$, and let $\mathcal{H}$ be the bipartite inclusion graph with parts $X=\binom{[n]\setminus\{x\}}{k+q-1}$ and $Y=\binom{[n]\setminus\{x\}}{k-1}$, where $A\in X$ is adjacent to $B\in Y$ if and only if $B\subset A$. All neighborhoods in the next theorem are taken in $\mathcal{H}$.

Balogh, Garcia, Li, and Wagner~\cite{BaloghGarciaLiWagner2024} used isoperimetric inequalities on $\mathcal{H}$ to obtain the following theorem, 
which follows from direct applications of the Lov\'asz~\cite{Lovasz1979} form of the Kruskal--Katona theorem~\cite{Kruskal1963,Katona2009}.

\begin{theorem}[Balogh, Garcia, Li, and Wagner~\cite{BaloghGarciaLiWagner2024}]\label{thm:isoperimetry}
Let $1\le q\le2+2\sqrt{k\ln k}$ and let $n=2k+q$. 
Fix $x\in[n]$ and let $\calA\subseteq\binom{[n]\setminus\{x\}}{k+q-1}$. Then
\begin{enumerate}[label=(\roman*),leftmargin=2.0em]
\item if $|\calA|\le\binom{n-2-c}{k+q-1}$ for a positive integer $c$, then
\begin{equation*}
|N(\calA)|\ge|\calA|\prod_{j=0}^{q-1}\left(1+\frac{c}{k+q-c-1-j}\right);
\end{equation*}
\item if $|\calA|\le\binom{k+q-1}{k-1}^{3}$, then
\begin{equation*}
|N(\calA)|\ge\frac{\binom{k+q-1}{k-1}}{(5e)^q}|\calA|.
\end{equation*}
\end{enumerate}
\end{theorem}

A bipartite graph $H$ with bipartition $X\cup Y$ is called $(d_X,d_Y)$-\emph{biregular} if every vertex in $X$ has degree $d_X$ and every vertex in $Y$ has degree $d_Y$.
For integers $a,g\ge1$, let
\begin{equation}\label{eq:container-class}
\calG(a,g)=\{\calA\subseteq X:\calA\text{ is a 2-linked set},\ |[\calA]|=a,\ |N(\calA)|=g\},
\end{equation}
where the neighborhood, closure, and $2$-linked property are taken with respect to $H$.

The following theorem is a biregular form of Sapozhenko's graph container method~\cite{Sapozhenko1987}. We use the formulation of Balogh, Garcia, Li, and Wagner~\cite[Theorem~5.1]{BaloghGarciaLiWagner2024}, restricted to the parts needed here.

\begin{theorem}[Balogh, Garcia, Li, and Wagner~\cite{BaloghGarciaLiWagner2024}]\label{thm:biregular-container}
Suppose $H$ is a $(d_X,d_Y)$-biregular graph with bipartition $X\cup Y$. 
Let $t=gd_Y-ad_X$ and, for $1\le\varphi\le d_Y-1$, let
\begin{equation*}
m_\varphi=\min\{|N(K)|:y\in Y,\ K\subseteq N(y),\ |K|>\varphi\}.
\end{equation*}
Let $1\le\psi\le\min\{d_X,d_Y\}-1$ and let $C>0$ satisfy $C\ln d_X/(\varphi d_X)<1$. 
Then there is a family $\calW(a,g)\subseteq2^X\times2^Y$ with
\begin{align*}
|\calW(a,g)|\le{}&|Y|\exp\!\left(\frac{54Cg\ln d_X\ln(d_Xd_Y)}{\varphi d_X}+\frac{54g\ln(d_Xd_Y)}{d_X^{Cm_\varphi/(\varphi d_X)}}+\frac{54t\ln d_Y\ln(d_Xd_Y)}{d_X(d_Y-\varphi)}\right)\notag\\
&\times\binom{3Cgd_Y\ln d_X/(\varphi d_X)}{\le3Ct\ln d_X/(\varphi d_X)}\binom{gd_Y}{\le t/((d_Y-\varphi)\psi)}\binom{gd_Xd_Y}{\le t/((d_X-\psi)\psi)}.
\end{align*}
Moreover, every $\calA\in\calG(a,g)$ is assigned a pair $(Q,R)\in\calW(a,g)$ such that
\begin{equation*}
[\calA]\subseteq Q,\quad R\subseteq N(\calA),\quad \text{and} \quad |Q|\le\frac{d_Y}{d_X}|R|+\frac{\psi t}{d_X}\left(\frac1{d_X-\psi}+\frac1{d_Y-\psi}\right).
\end{equation*}
\end{theorem}

We shall use the following rounded form of the standard binomial-tail estimate. 

\begin{lemma}\label{lem:rounded-binomial-tail}
If $z\geq u\geq2$, then $\binom{z}{\leq u}
\leq
\left(3ez/u\right)^u$.
\end{lemma}
\begin{proof}
We let $N=\lceil z\rceil$ and $r=\lfloor u\rfloor$. Since $1\leq r\leq N$, the standard binomial-tail estimate gives
\begin{equation*}
\sum_{j=0}^{r}\binom Nj
\leq
\left(\frac{eN}{r}\right)^r.
\end{equation*}
Since $z\geq u\geq2$, we have $N\leq z+1\leq3z/2$ and $r\geq u-1\geq u/2$. Hence $N/r\leq3z/u$. Moreover, $r\leq u$ and $3ez/u>1$. Therefore
\begin{equation*}
\binom{z}{\leq u}
\leq
\left(\frac{eN}{r}\right)^r
\leq
\left(\frac{3ez}{u}\right)^r
\leq
\left(\frac{3ez}{u}\right)^u,
\end{equation*}
as required.
\end{proof}

\begin{lemma}\label{lem:connected}
Let $J$ be a graph on $n$ vertices with maximum degree at most $\Delta\ge2$. 
For every integer $a\ge1$, the number of connected $a$-vertex subsets of $J$ is at most $n(4\Delta)^{a-1}$.
\end{lemma}

\begin{proof}
The proof uses the standard rooted-tree encoding. 
We include the details for completeness. 
Choose the image of the root in $n$ ways. Every connected $a$-vertex set has a rooted spanning tree. 
Ordering the children at each vertex produces a rooted plane tree, that is, a rooted tree in which the children of each vertex are ordered, and there are at most $4^{a-1}$ rooted plane trees on $a$ vertices. 
Once the root image and the plane tree are fixed, every other vertex has at most $\Delta$ possible images along the edge to its parent. 
Counting non-injective maps only enlarges the total, so the number of connected $a$-vertex sets is at most $n4^{a-1}\Delta^{a-1}=n(4\Delta)^{a-1}$.
\end{proof}

\section{Reduction to fixed stars}\label{sec:global}
Throughout the rest of the paper, $n=2k+q$ with $3\le q<100\ln k$, and all statements are uniform in this range as $k \to \infty$.
For sufficiently large $k$, this range lies inside $q\le2+2\sqrt{k\ln k}$, and also $n<3k$ and $k+q<2k$.

Let $\gamma=1/(2\cdot4^9)$.
Let $H_2(p)=-p\log p-(1-p)\log(1-p)$ be the binary entropy function. 
Choose a fixed $\zeta>0$ so small that
\begin{equation*}\label{eq:zeta}
2H_2(\zeta)<\frac\gamma2.
\end{equation*}
We also use the constants $\gamma$ and $\zeta$ throughout the rest of the paper.

The deterministic pruning argument below is in the spirit of the Kleitman--Winston method~\cite{KleitmanWinston1982}.

\begin{lemma}\label{lem:globalcontainer}
Fix $\eta>0$ and set $L=\eta q\binom{k+q}{q}/k$.
For a set $U\subseteq V(KG(n,k))$, let $e(U)$ denote the number of edges
in the subgraph of $KG(n,k)$ induced by $U$. For all sufficiently large
$k$, there is a family $\mathfrak C$ of subsets of $V(KG(n,k))$ with the
following properties:
\begin{enumerate}
\item every independent set in $KG(n,k)$ is contained in some $C\in\mathfrak C$;
\item every $C\in\mathfrak C$ satisfies $e(C)\leq \frac{L }{2}\binom {n}{k}$;
\item we have $$\log|\mathfrak C|\leq \frac{\binom {n}{k}}{L}\log(eL)=o\!\left(\frac{q}{k}\binom{n-1}{k-1}\right).$$
\end{enumerate}
\end{lemma}

\begin{proof}
For convenience, we write $$G=KG(n,k)\quad\text{and}\quad V=V(G).$$
Fix a total ordering $\prec$ of the vertices of $G$. This
ordering will be used to make every choice in the construction unique.

Let $I\subseteq V$ be an independent set. We construct two sequences of sets $$A_0,A_1,\ldots\quad\text{and}\quad P_0,P_1,\ldots.$$
The set $A_i$ is the set of vertices that are still active after $i$ steps, while $P_i$ is the fingerprint recorded during the first $i$ steps.
Start with $A_0=V$ and $P_0=\varnothing$.
Suppose that $A_i$ and $P_i$ have already been defined. 
If the maximum degree of the induced graph $G[A_i]$ is less than $L$, then we stop. 
Otherwise, let $v_i$ be the first vertex, with respect to $\prec$, among all vertices of maximum degree in $G[A_i]$.

We now update $A_i$ and $P_i$ as follows.
If $v_i\in I$, put $P_{i+1}=P_i\cup\{v_i\}$ and $A_{i+1}=A_i\setminus\bigl(\{v_i\}\cup N_{G[A_i]}(v_i)\bigr)$.
If $v_i\notin I$, put $P_{i+1}=P_i$ and $A_{i+1}=A_i\setminus\{v_i\}$.
At least one vertex is removed at each step, so the process terminates. Denote the terminal sets by $P(I)$ and $A(I)$.
For every $i$, we have $P_i\subseteq I$ and $I\subseteq P_i\cup A_i$. Indeed, if $v_i\notin I$, only a vertex outside $I$ is removed, while if $v_i\in I$, then $v_i$ is added to $P_{i+1}$ and none of its neighbors belongs to $I$, since $I$ is independent. Hence $I\subseteq P(I)\cup A(I)$.
The set $P(I)$ is independent, and whenever a vertex enters $P(I)$, all its neighbors still present are deleted. 
Hence no edge lies inside $P(I)$ or between $P(I)$ and the set $A(I)$, so 
$$ e\bigl(P(I)\cup A(I)\bigr)=e(A(I))=\frac{1}{2} \sum_{v\in A(I)}d_{G[A(I)]}(v)<\frac{ L|A(I)|}{2}\leq \frac{ L|V|}{2}= \frac{L }{2}\binom {n}{k}.$$
Each vertex inserted into $P(I)$ deletes at least $L$ vertices, whence $|P(I)|\le\binom nk/L$.
We next show that the final active set is determined by the fingerprint. Given a set $P$, start with $A=V(G)$ and repeat the same deterministic process, using membership in $P$ in place of membership in $I$. Whenever the chosen vertex $v$ belongs to $P$, remove $v$ and all its current neighbors; otherwise, remove only $v$. Denote the final active set by $A(P)$. If $P=P(I)$, the fixed ordering ensures that this replay makes the same choice at every step as the original process, and hence $A(P(I))=A(I)$.

Let $$\mathfrak C=\{P\cup A(P):P=P(I)\text{ for some independent set }I\text{ of }G\}.$$
The argument above shows that every independent set $I$ is contained in $P(I)\cup A(P(I))$, which belongs to $\mathfrak C$,
and every $C\in\mathfrak C$ satisfies $e(C)\leq \frac{L }{2}\binom {n}{k}$.
Since each fingerprint satisfies $|P(I)|\le\binom nk/L$, we deduce that
$$|\mathfrak C| \leq \sum_{j=0}^{\lfloor \binom{n}{k}/L\rfloor}\binom{\binom{n}{k}}{j}.$$

It remains to estimate this upper bound.
Let $p=1/L$. For all sufficiently large $k$, we have $L\geq2$, and hence $0<p\leq1/2$. 
Recall that $H_2(p)=-p\log p-(1-p)\log(1-p)$.
By the binomial theorem, 
$$1=\sum_{j=0}^{\binom{n}{k}}\binom{\binom{n}{k}}{j}p^j(1-p)^{\binom{n}{k}-j}.$$
For every $j\leq p\binom{n}{k}$, since $p\leq 1-p$, we have $p^j(1-p)^{\binom{n}{k}-j}\geq p^{p\binom{n}{k}}(1-p)^{(1-p)\binom{n}{k}}=2^{-\binom{n}{k}H_2(p)}$.
It follows that
$$1\geq 2^{-\binom{n}{k}H_2(p)}\sum_{j=0}^{\left\lfloor p\binom{n}{k}\right\rfloor}\binom{\binom{n}{k}}{j},$$
and hence
$$\sum_{j=0}^{\left\lfloor p\binom{n}{k}\right\rfloor}\binom{\binom{n}{k}}{j}\leq 2^{\binom{n}{k}H_2(p)}.$$
Moreover, using $-(1-p)\log(1-p)\leq p\log e$, we obtain $H_2(p)\leq p\log\frac{1}{p}+p\log e=p\log\frac{e}{p}$.
Substituting $p=1/L$ yields
$$\sum_{j=0}^{\left\lfloor \binom{n}{k}/L\right\rfloor}\binom{\binom{n}{k}}{j}
\leq 2^{\binom{n}{k}H_2(1/L)}
\leq 2^{\frac{1}{L}\binom{n}{k}\log(eL)}.$$
Therefore, $\log|\mathfrak C|\leq (\binom nk/L)\log(eL)$.
A simple calculation shows that
\begin{equation*}
\frac{(\binom nk/L)\log(eL)}{(q/k)\binom{n-1}{k-1}}=\frac{nk\log(eL)}{\eta q^2\binom{k+q}{q}}.
\end{equation*}
Since $3\le q=n-2k<100\ln k$, $\binom{k+q}{q}\ge\binom{k+3}{3}$ and $\log(eL)=O(q\log k)$, so the last expression is $O(\log k/k)=o(1)$.
This proves the lemma.
\end{proof}

The following theorem shows that all but $o\bigl(2^{\binom{n-1}{k-1}}\bigr)$ intersecting families
have at most $\zeta\binom{n-1}{k-1}$ members outside one of the full stars.

\begin{theorem}\label{thm:reduce}
All but $o\bigl(2^{\binom{n-1}{k-1}}\bigr)$ intersecting families $\calF\subseteq\binom{[n]}k$ satisfy
\begin{equation*}
|\calF\setminus\calS_x|\le\zeta\binom{n-1}{k-1}
\end{equation*}
for at least one $x\in[n]$.
\end{theorem}

\begin{proof}
Choose fixed $\delta,\eta>0$ sufficiently small that
\begin{equation}\label{eq:constants}
9\eta\le(20C_{\rm DT})^{-2}, \quad 3C_{\rm DT}(\delta+15\eta)\le\zeta, \quad\text{and} \quad 6(\delta+9\eta)\le(20C_{\rm DT})^{-2},
\end{equation}
where $C_{\rm DT}$ is given by Theorem~\ref{thm:DT}.
Apply Lemma~\ref{lem:globalcontainer} with the fixed value of $\eta$. 
Each $C\in \mathfrak C$ satisfying $|C|\le\binom{n-1}{k-1}-\delta\frac qk\binom{n-1}{k-1}$ contains at most
$2^{|C|}\leq 2^{\binom{n-1}{k-1}-\delta(q/k)\binom{n-1}{k-1}}$ independent sets. 
Therefore, by the union bound, all such containers contain at most
$|\mathfrak C|2^{\binom{n-1}{k-1}-\delta(q/k)\binom{n-1}{k-1}}$ independent sets. 
Since $\log|\mathfrak C|=o\left(\frac{q}{k}\binom{n-1}{k-1}\right)$ by Lemma~\ref{lem:globalcontainer},
this number is at most 
$$ 2^{\binom{n-1}{k-1}-\delta(q/k)\binom{n-1}{k-1} +o\left((q/k)\binom{n-1}{k-1}\right)}.$$
Moreover, $\frac{q}{k}\binom{n-1}{k-1}\to\infty$, so
$$
2^{-\delta(q/k)\binom{n-1}{k-1}
+o\left((q/k)\binom{n-1}{k-1}\right)}=o(1).
$$
Hence, all such $C$ with $|C|\le\binom{n-1}{k-1}-\delta\frac qk\binom{n-1}{k-1}$ contain altogether at most $o\!\left(2^{\binom{n-1}{k-1}}\right)$ independent sets.

Consider any remaining $C\in \mathfrak C$, and write $|C|=(1-\alpha)\binom{n-1}{k-1}$. 
If $|C|\le\binom{n-1}{k-1}$, then $0\leq\alpha<\delta q/k$. 
If $|C|>\binom{n-1}{k-1}$, Theorem~\ref{thm:spectrum} and $e(C)\le L\binom nk/2$ give
\begin{equation*}
\frac{|C|-\binom{n-1}{k-1}}{\binom{n-1}{k-1}}\le\frac{L}{\binom{k+q}{q}+\binom{k+q-1}{q}}\left(\frac nk\right)^2\le9\eta\frac qk.
\end{equation*}
Thus
\begin{equation}\label{eq:alpha-bound}
|\alpha|\le(\delta+9\eta)\frac qk.
\end{equation}
The number of disjoint pairs in $C$ is at most $L\binom nk/2$, and therefore
\begin{equation}\label{eq:beta-bound}
\beta:=\frac{e(C)}{\binom{n-1}{k-1}\binom{k+q-1}{q}}\le\eta\frac{qn(k+q)}{2k^3}\le3\eta\frac qk.
\end{equation}
The hypotheses of Theorem~\ref{thm:DT} follow from the combination of \eqref{eq:alpha-bound}, \eqref{eq:beta-bound}, the last two inequalities in \eqref{eq:constants}, and $n/k<3$. Hence some star $\calS_x$ satisfies
\begin{equation*}
|C\mathbin\triangle\calS_x|\le C_{\rm DT}(|\alpha|+2\beta)\frac nq\binom{n-1}{k-1}\le3C_{\rm DT}(\delta+15\eta)\binom{n-1}{k-1}\le\zeta\binom{n-1}{k-1}.
\end{equation*}
Every independent set $\calF\subseteq C$ then has $|\calF\setminus\calS_x|\le\zeta\binom{n-1}{k-1}$. Since independent sets of the Kneser graph are exactly intersecting families, the theorem follows.
\end{proof}

\section{Proof of Theorem~\ref{thm:main}}\label{sec:main-thm}

It remains to count non-trivial families near a fixed star. We prove the following estimate.

\begin{theorem}\label{thm:nearstar}
For every fixed $x\in[n]$, the number of non-trivial intersecting families $\calF$ satisfying 
$|\calF\setminus\calS_x|\le\zeta\binom{n-1}{k-1}$ is at most $2^{(1-\gamma/2)\binom{n-1}{k-1}}+2^{\binom{n-1}{k-1}-\binom{n-k-1}{k-1}/10}$.
\end{theorem}

Before we move to the technical proof of Theorem~\ref{thm:nearstar}, we first show how this theorem 
will be applied to complete the proof of Theorem~\ref{thm:main}.

\begin{proof}[\textnormal{\textbf{Proof of Theorem~\ref{thm:main} using Theorem~\ref{thm:nearstar}}}]
If $n-2k\ge100\ln k$, the conclusion follows from Theorem~\ref{thm:BaloghGarciaLiWagner2024}.
We may suppose that $3\le n-2k<100\ln k$. 
Theorem~\ref{thm:reduce} leaves only $o\bigl(2^{\binom{n-1}{k-1}}\bigr)$ exceptional intersecting families. 
Every remaining family satisfies $|\calF\setminus\calS_x|\leq\zeta\binom{n-1}{k-1}$ for at least one of the $n$ stars, 
and Theorem~\ref{thm:nearstar} bounds the number of non-trivial families near each fixed star by
$$2^{(1-\gamma/2)\binom{n-1}{k-1}}+2^{\binom{n-1}{k-1}-\binom{n-k-1}{k-1}/10}.$$
Since $3\le n-2k<100\ln k$, we have $n<3k$ and $\binom{n-1}{k-1}\geq\binom{n-k-1}{k-1}\geq\binom{k+2}{3}$. Therefore,
$$n2^{-\gamma\binom{n-1}{k-1}/2}+n2^{-\binom{n-k-1}{k-1}/10}=o(1)$$ uniformly in this range. 
Together with Theorem~\ref{thm:reduce}, this gives $I(n,k,\geq 1)=o\bigl(2^{\binom{n-1}{k-1}}\bigr)$.
Thus, the two ranges cover every $n\ge2k+3$.

There are at most $n2^{\binom{n-1}{k-1}}$ trivial intersecting families. 
By the first two terms of inclusion--exclusion, their number is at least $n2^{\binom{n-1}{k-1}}-\binom n2 2^{\binom{n-2}{k-2}}=(n-o(1))2^{\binom{n-1}{k-1}}$, 
since $\binom n2 2^{-\binom{n-2}{k-1}}=o(1)$, as $\binom{n-2}{k-1}\geq n-2$. 
Hence the number of trivial intersecting families is $(n+o(1))2^{\binom{n-1}{k-1}}$. 
Combining this with the estimate for non-trivial families yields $I(n,k)=(n+o(1))2^{\binom{n-1}{k-1}}$ and $I(n,k,\geq 1)=o(I(n,k))$.
\end{proof}

\section{Proof of Theorem~\ref{thm:nearstar}}

Fix $x\in[n]$ and let
\begin{equation*}
X=\binom{[n]\setminus\{x\}}{k+q-1}\quad \text{and} \quad Y=\binom{[n]\setminus\{x\}}{k-1}.
\end{equation*}
Since $n=2k+q$, we have
\begin{equation*}
|X|=\binom{n-1}{k+q-1}=\binom{n-1}{k}=\frac{k+q}{k}\binom{n-1}{k-1}\quad \text{and}\quad |Y|=\binom{n-1}{k-1}.
\end{equation*}
Let $H$ be the bipartite graph with parts $X$ and $Y$, where $A\in X$ is adjacent to $B\in Y$ if and only if $B\subset A$. 
The degrees of the vertices in $X$ and $Y$, respectively, are
\begin{equation*}
d_X=\binom{k+q-1}{k-1}\quad \text{and} \quad d_Y=\binom{k+q}{q}=\frac{k+q}{k} d_X.
\end{equation*}
Thus $H$ is $(d_X,d_Y)$-biregular.

Given an intersecting family $\calF$, define
\begin{equation*}
\calA=\{([n]\setminus\{x\})\setminus F:F\in\calF,\ x\notin F\}\quad \text{and} \quad \calB=\{F\setminus\{x\}:F\in\calF,\ x\in F\}.
\end{equation*}
Then $\calA\subseteq X$ and $\calB\subseteq Y$.
We claim that $\calA\cup\calB$ is independent in $H$. 
To see this, take $A\in\calA$ and $B\in\calB$, and choose $F_A,F_B\in\calF$ such that $x\notin F_A$, $A=([n]\setminus\{x\})\setminus F_A$, $x\in F_B$, and $B=F_B\setminus\{x\}$.
If $AB$ is an edge in $H$, then $B\subset A=([n]\setminus\{x\})\setminus F_A$.
Thus $B\cap F_A=\varnothing$.
Since $F_B=B\cup\{x\}$ and $x\notin F_A$, $F_A\cap F_B=F_A\cap(B\cup\{x\})=(F_A\cap B)\cup(F_A\cap\{x\})=\varnothing$. 
This contradicts the assumption that $\calF$ is intersecting.
Hence $\calA\cup\calB$ is independent in $H$. 
Thus, once $\calA$ is fixed, the set $\calB$ must be contained in $Y\setminus N(\calA)$, and there are at most
\begin{equation}\label{eq:Bchoices}
2^{\binom{n-1}{k-1}-|N(\calA)|}
\end{equation}
choices for $\calB$.

Let
\begin{equation*}\label{eq:T0}
T_0=\binom{n-10}{k+q-1}.
\end{equation*}
Applying Theorem~\ref{thm:isoperimetry} with $c=8$, we get
\begin{equation}\label{eq:iso-large}
|\calA|\le T_0\quad\Longrightarrow\quad |N(\calA)|\ge|\calA|\left(1+\frac8{k+q-9}\right)^q,
\end{equation}
and
\begin{equation}\label{eq:iso-small}
|\calA|\le d_X^3\quad\Longrightarrow\quad |N(\calA)|\ge\frac{d_X}{(5e)^q}|\calA|.
\end{equation}

We now consider two cases according to the size of $[\calA]$.

\begin{claim}\label{claim:largeclosure}
The number of intersecting families $\calF$ satisfying $|\calF\setminus\calS_x|\le\zeta\binom{n-1}{k-1}$ and $|[\calA]|>T_0$ is at most
\begin{equation*}
2^{(1-\gamma/2)\binom{n-1}{k-1}}.
\end{equation*}
\end{claim}

\begin{proof}
Since $N([\calA])=N(\calA)$, edge counting between $[\calA]$ and its neighborhood gives
\begin{equation*}
|N(\calA)|\ge\frac{d_X}{d_Y}|[\calA]|>\frac{k}{k+q}T_0.
\end{equation*}
Since $n=2k+q$ and $q=o(k)$, for all sufficiently large $k$ we have
\begin{equation*}
\frac{T_0}{|X|}=\frac{\binom{n-10}{k+q-1}}{\binom{n-1}{k+q-1}}=\frac{\binom{n-10}{k-9}}{\binom{n-1}{k}}=\prod_{i=0}^{8}\frac{k-i}{n-1-i}\ge4^{-9}.
\end{equation*}
Together with the definition of $\gamma$, $k/(k+q)\ge1/2$ (as $k+q<2k$) and $|X|\ge\binom{n-1}{k-1}$, this yields $|N(\calA)|\ge\gamma\binom{n-1}{k-1}$.

The condition $|\calF\setminus\calS_x|\le\zeta\binom{n-1}{k-1}$ implies $|\calA|\le\zeta\binom{n-1}{k-1}$. 
The standard entropy estimate $\sum_{j\le pM}\binom Mj\le2^{H_2(p)M}$, valid for $0<p\le1/2$, and $|X|\le2\binom{n-1}{k-1}$ 
show that the number of possible $\calA$ is at most $2^{2H_2(\zeta)\binom{n-1}{k-1}}$. 
For each one, \eqref{eq:Bchoices} gives at most $2^{(1-\gamma)\binom{n-1}{k-1}}$ choices for $\calB$. 
This together with $2H_2(\zeta)<\gamma/2$ yields the result.
\end{proof}

\begin{claim}\label{claim:localcontainer}
There is a function $\theta=\theta(k,q)$ such that
\begin{equation*}
\frac1{k\sqrt{d_X}}\leq\theta=o\!\left(\frac qk\right)
\end{equation*}
and the following holds. 
If $a\le T_0$ and $g\ge d_X^3$, then there is a family $\calW(a,g)$ of at most $2^{\theta g}$ pairs $(Q,R)$ with $Q\subseteq X$ and $R\subseteq Y$ 
such that every $\calA\in\calG(a,g)$ is assigned a pair satisfying
\begin{equation*}
[\calA]\subseteq Q,\quad R\subseteq N(\calA),\quad |Q|\le\left(1+\frac qk\right)|R|+\theta g.
\end{equation*}
\end{claim}

\begin{proof}
Let
\begin{equation*}
t=d_X\left(\left(1+\frac qk\right)g-a\right).
\end{equation*}
Applying \eqref{eq:iso-large} to $[\calA]$, whose neighborhood is $N(\calA)$, gives $g>a$. 
Hence,
\begin{equation}\label{eq:gdt}
\frac{gd_X}{t}=\frac{g}{(1+q/k)g-a}\le\frac{k}{q}\le d_X^{1/q},
\end{equation}
where the last inequality follows after an easy calculation from $d_X=\binom{k+q-1}{k-1}=\binom{k+q-1}{q}\ge(k/q)^q$.

Apply Theorem~\ref{thm:biregular-container} to $H$ with $\varphi=d_X/2$, $\psi=\sqrt{d_X\ln d_X}$ and $C=10(5e)^q$. 
If $K\subseteq N(y)$ with $y\in Y$ and $|K|>\varphi$, then $|K|\le d_Y<2d_X<d_X^3$, 
so Theorem~\ref{thm:isoperimetry} gives
\begin{equation*}
|N(K)|\geq \frac{d_X}{(5e)^q} |K| >\frac{\varphi d_X}{(5e)^q}.
\end{equation*}
Let $m_\varphi=\min\{|N(K)|:y\in Y,\ K\subseteq N(y),\ |K|>\varphi\}$. 
Clearly, $m_\varphi >\varphi d_X/(5e)^q$.
Thus $Cm_\varphi/(\varphi d_X)>10$. 
Also $C\ln d_X/(\varphi d_X)=20(5e)^q\ln d_X/d_X^2=o(1)$.
It follows from Theorem~\ref{thm:biregular-container} that there is a family $\calW(a,g)\subseteq2^X\times2^Y$ with
\begin{align}\label{eq:quoted-container-size}
|\calW(a,g)|\le{}&|Y|\exp\!\left(\frac{54Cg\ln d_X\ln(d_Xd_Y)}{\varphi d_X}+\frac{54g\ln(d_Xd_Y)}{d_X^{Cm_\varphi/(\varphi d_X)}}+\frac{54t\ln d_Y\ln(d_Xd_Y)}{d_X(d_Y-\varphi)}\right)\notag\\
&\times\binom{3Cgd_Y\ln d_X/(\varphi d_X)}{\le3Ct\ln d_X/(\varphi d_X)}\binom{gd_Y}{\le t/((d_Y-\varphi)\psi)}\binom{gd_Xd_Y}{\le t/((d_X-\psi)\psi)}.
\end{align}
Moreover, every $\calA\in\calG(a,g)$ is assigned a pair $(Q,R)\in\calW(a,g)$ such that
\begin{equation}\label{eq:quoted-container-structure}
[\calA]\subseteq Q,\quad R\subseteq N(\calA),\quad |Q|\le\frac{d_Y}{d_X}|R|+\frac{\psi t}{d_X}\left(\frac1{d_X-\psi}+\frac1{d_Y-\psi}\right).
\end{equation}

We call a pair $(Q,R)$ satisfying the conclusion of the claim a \emph{good pair}.
We first estimate the bound \eqref{eq:quoted-container-size} term by term. 

Since $g>a$ and $q<k$, we have
\begin{equation}\label{eq:t-bounds}
\frac qk d_Xg<t=d_X\left(\left(1+\frac qk\right)g-a\right)<2d_Xg.
\end{equation}
Recall that $d_Y<2d_X$, $\varphi=d_X/2$, $d_Y-\varphi\ge d_X/2$, and $Cm_\varphi/(\varphi d_X)>10$. 
Therefore
$$\frac{54Cg\ln d_X\ln(d_Xd_Y)}{\varphi d_X} \leq \frac{54Cg\ln d_X \cdot 3 \ln d_X}{d^2_X/2}=324\frac{Cg\ln^2 d_X}{d_X^2},$$
$$\frac{54g\ln(d_Xd_Y)}{d_X^{Cm_\varphi/(\varphi d_X)}}\leq \frac{54g\ln(d_Xd_Y)}{d_X^{10}}\leq 162\frac{g\ln d_X}{d_X^{10}}\leq 162\frac{Cg\ln^2 d_X}{d_X^2},$$
and
$$\frac{54t\ln d_Y\ln(d_Xd_Y)}{d_X(d_Y-\varphi)}\leq \frac{54t\cdot 2\ln d_X \cdot 3\ln d_X}{d_X \cdot d_X/2}=648\frac{t\ln^2 d_X}{d_X^2}\leq 648\frac{Ct\ln^2 d_X}{d_X^2}.$$
The natural logarithm of the exponential factor in \eqref{eq:quoted-container-size} is at most
\begin{equation*}
648\left(\frac{Cg\ln^2d_X}{d_X^2}+\frac{Ct\ln^2d_X}{d_X^2}\right).
\end{equation*}

We first verify that Lemma~\ref{lem:rounded-binomial-tail} is applicable to all three binomial factors. Since $\psi=\sqrt{d_X\ln d_X}<d_X$ for all sufficiently large $k$, equations \eqref{eq:t-bounds}, $g\geq d_X^3$, and $\varphi=d_X/2$ give
\begin{align*}
\frac{3Ct\ln d_X}{\varphi d_X}>\frac{6Cq\,d_X^2\ln d_X}{k}, \quad
\frac{t}{(d_Y-\varphi)\psi}>\frac{q\,d_X^2}{2k},\quad \text{and} \quad
\frac{t}{(d_X-\psi)\psi}>\frac{q\,d_X^2}{k}.
\end{align*}
Since $q\geq3$ and $d_X=\binom{k+q-1}{q}\geq\binom{k+2}{3}$, all three quantities are at least $2$ for sufficiently large $k$. Moreover, $t=gd_Y-ad_X\leq gd_Y$, and the remaining denominator factors are at least $1$. Thus, in each of the three binomial factors, the upper parameter is at least the corresponding cutoff parameter.

Since $q\geq3$ and $d_X\to\infty$, we may also suppose that
\begin{equation*}
\ln\left(6e d_X^{1/q}\right)\leq\ln d_X,
\quad
\ln\left(12e d_X^{2+1/q}\right)\leq5\ln d_X,
\quad \text{and} \quad
\ln\left(6e d_X^{3+1/q}\right)\leq5\ln d_X.
\end{equation*}

For the first binomial factor, Lemma~\ref{lem:rounded-binomial-tail}, $\varphi=d_X/2$, $d_Y<2d_X$, and \eqref{eq:gdt} give
\begin{align*}
\ln \binom{3Cgd_Y\ln d_X/(\varphi d_X)}{\leq 3Ct\ln d_X/(\varphi d_X)}
\leq\frac{3Ct\ln d_X}{\varphi d_X}\ln\left(\frac{3egd_Y}{t}\right)
\leq\frac{6Ct\ln d_X}{d_X^2}\ln\left(6e d_X^{1/q}\right)
\leq\frac{6Ct\ln^2 d_X}{d_X^2}.
\end{align*}

For the second binomial factor, Lemma~\ref{lem:rounded-binomial-tail}, $d_Y-\varphi\geq d_X/2$, $d_Y<2d_X$, $\psi<d_X$, and \eqref{eq:gdt} give
\begin{align*}
\ln\binom{gd_Y}{\leq t/((d_Y-\varphi)\psi)}
\leq\frac{t}{(d_Y-\varphi)\psi}\ln\left(\frac{3egd_Y(d_Y-\varphi)\psi}{t}\right)
\leq\frac{2t}{d_X\psi}\ln\left(12e d_X^{2+1/q}\right)
\leq 10\frac{t\sqrt{\ln d_X}}{d_X^{3/2}}.
\end{align*}

For the third binomial factor, since $\psi=o(d_X)$, we may suppose that $d_X-\psi\geq d_X/2$. 
Hence Lemma~\ref{lem:rounded-binomial-tail}, $d_Y<2d_X$, $\psi<d_X$, and \eqref{eq:gdt} give
\begin{align*}
\ln\binom{gd_Xd_Y}{\leq t/((d_X-\psi)\psi)}
\leq\frac{t}{(d_X-\psi)\psi}\ln\left(\frac{3egd_Xd_Y(d_X-\psi)\psi}{t}\right)
\leq\frac{2t}{d_X\psi}\ln\left(6e d_X^{3+1/q}\right)
\leq10\frac{t\sqrt{\ln d_X}}{d_X^{3/2}}.
\end{align*}

Combining these estimates, converting from natural to binary logarithms, and using $|Y|\leq2^n$, we obtain
\begin{equation}\label{eq:Wsimple}
\log|\calW(a,g)|\le n+1000\left(\frac{Ct\ln^2d_X}{d_X^2}+\frac{Cg\ln^2d_X}{d_X^2}+\frac{t\sqrt{\ln d_X}}{d_X^{3/2}}\right).
\end{equation}

The structural bound \eqref{eq:quoted-container-structure}, together with $\psi=\sqrt{d_X\ln d_X}$ and \eqref{eq:t-bounds}, yields
\begin{equation*}
|Q|\le\left(1+\frac qk\right)|R|+\frac{4\psi t}{d_X^2}\le\left(1+\frac qk\right)|R|+8g\sqrt{\frac{\ln d_X}{d_X}}.
\end{equation*}
Choose an absolute $C_0$ large enough and set
\begin{equation*}
\theta=C_0\left(\frac{n}{d_X^3}+\frac{C\ln^2d_X}{d_X}+\sqrt{\frac{\ln d_X}{d_X}}\right)+\frac1{k\sqrt{d_X}}.
\end{equation*}
Then $|Q|\le\left(1+q/k\right)|R|+\theta g$.
Since $g\ge d_X^3$, \eqref{eq:t-bounds} and \eqref{eq:Wsimple} imply $|\calW(a,g)|\le2^{\theta g}$.

It is enough to show that $\theta=o(q/k)$. We have $d_X\ge\binom{k+2}{3}$, $d_X\ge(k/q)^q$ and $\ln d_X=O(q\ln k)$. Since $5eq/k<1$ for large $k$,
\begin{equation*}
\frac{k}{q}\frac{10(5e)^q\ln^2d_X}{d_X}\le10\frac{k}{q}\left(\frac{5eq}{k}\right)^q\ln^2d_X=O\!\left(\frac{\ln^6k}{k^2}\right).
\end{equation*}
Furthermore,
\begin{equation*}
\frac{k}{q}\sqrt{\frac{\ln d_X}{d_X}}=O\!\left(\frac{\ln k}{\sqrt{k}}\right),\quad \frac{k}{q}\frac{n}{d_X^3}=o(1),\quad \text{and} \quad \frac{k}{q}\frac1{k\sqrt{d_X}}=o(1).
\end{equation*}
This proves $1/(k\sqrt{d_X})\leq\theta=o\!\left(q/k\right)$.
\end{proof}

Recall that
$$\calG(a,g)=\{\calA\subseteq X:\calA\text{ is a 2-linked set},\ |[\calA]|=a,\ |N(\calA)|=g\}.$$
\begin{claim}\label{claim:component}
For sufficiently large $k$, if $a\le T_0$, then
\begin{equation*}
|\calG(a,g)|\le
\begin{cases}
2^{g/2},&d_X\le g\le d_X^3,\\
2^{g(1-1/(k\sqrt{d_X}))},&g>d_X^3.
\end{cases}
\end{equation*}
\end{claim}

\begin{proof}
It is obvious that if $g<d_X$, then $\calG(a,g)=\emptyset$.
We first suppose that $d_X\leq g\leq d_X^3$. 
Let $J$ be the graph on $X$ in which two vertices are adjacent whenever their distance in $H$ is at most two. 
We first show that $[\calA]$ is connected in $J$. 
Since $\calA$ is $2$-linked, the graph $J[\calA]$ is connected. 
For each $C\in[\calA]\setminus\calA$, choose $B\in N(C)$. 
Since $N(C)\subseteq N(\calA)$, there is some $A\in\calA$ adjacent to $B$. 
Thus $C$ and $A$ have distance two in $H$, and hence they are adjacent in $J$. 
It follows that $J[[\calA]]$ is connected, equivalently, that $[\calA]$ is 2-linked in $H$.

Each vertex of $X$ has $d_X$ neighbors in $Y$, and each of these vertices has at most $d_Y$ neighbors in $X$. 
Therefore, $\Delta(J)\leq d_Xd_Y$.
By Lemma~\ref{lem:connected}, the number of possible sets $[\calA]$ of size $a$ is at most $|X|(4d_Xd_Y)^{a-1}$.
Once $[\calA]$ is fixed, there are at most $2^a$ choices for $\calA\subseteq[\calA]$. Since $d_Y<2d_X$ and $|X|\leq2^n$, for all sufficiently large $k$ we obtain
$$|\calG(a,g)|\leq |X|(4d_Xd_Y)^{a-1}2^a\leq 2^{n+12a\ln d_X}.$$

By the definition of the closure, $N([\calA])=N(\calA)$.
Indeed, $\calA\subseteq[\calA]$ yields $N(\calA)\subseteq N([\calA])$, while the definition of $[\calA]$ gives the reverse inclusion. 
By \eqref{eq:iso-large}, we obtain $g=|N([\calA])|>|[\calA]|=a$.
Thus $a<g\leq d_X^3$, and \eqref{eq:iso-small} gives $g\geq\frac{d_Xa}{(5e)^q}$, so $a\leq\frac{(5e)^qg}{d_X}$.
Moreover, since $\calA\neq\varnothing$, we have $g=|N(\calA)|\geq d_X$. Consequently,
$$\frac{n+12a\ln d_X}{g}\leq\frac{n}{d_X}+\frac{12(5e)^q\ln d_X}{d_X}.$$
Since $3\leq q=n-2k<100\ln k$, we have $n<3k$ and $d_X\geq\binom{k+2}{3}$, so $n/d_X=o(1)$. Also,
$$
d_X=\binom{k+q-1}{q}\geq\left(\frac{k}{q}\right)^q,$$
and hence
$$\frac{(5e)^q\ln d_X}{d_X}\leq\left(\frac{5eq}{k}\right)^q\ln d_X=o(1).$$
Therefore, $n+12a\ln d_X=o(g)$.
For all sufficiently large $k$, this implies $n+12a\ln d_X\leq g/2$, and hence
$$
|\calG(a,g)|\leq2^{g/2}.
$$

Now we suppose that $g>d_X^3$. 
Fix a pair $(Q,R)$ given by Claim~\ref{claim:localcontainer},
and define $\xi\in\mathbb R$ by $|Q|=(1+\xi/k)g$. 
From Claim~\ref{claim:localcontainer},
\begin{equation}\label{eq:xi-upper}
\frac\xi k\le\frac qk+\theta.
\end{equation}
If $\xi/k<-4\theta$, then every $\calA\in\calG(a,g)$ assigned to
$(Q,R)$ satisfies
$\calA\subseteq[\calA]\subseteq Q$. Hence at most
$2^{|Q|}\leq2^{(1-4\theta)g}$ such families are assigned to $(Q,R)$.
Since there are at most $2^{\theta g}$ good pairs, this case contributes at most
\[
2^{(1-3\theta)g}
\le
2^{g\left(1-\frac{3}{k\sqrt{d_X}}\right)},
\]
where we used $\theta\ge1/(k\sqrt{d_X})$.

It remains to consider $\xi/k\ge-4\theta$.  
Discard any pair $(Q,R)$ to which no set is assigned. 
For each remaining pair, choose one assigned set $\calA_0$ and define $\calA^*=[\calA_0]$ and $N^*=N(\calA_0)$. 
For any other assigned set $\calA$, let $h=g-|R|$ and $m=|Q|-a$.
Clearly, $h\geq 0$ and $m\geq 0$.
Both $N(\calA)$ and $N^*$ have size $g$ and contain $R$, and hence $|N^*\setminus R|=h$.  So after $R$ and $N^*$ are fixed, the set $(N(\calA)\cap N^*)\setminus R$ is an arbitrary subset of $N^* \setminus R$, and hence has at most $2^h$ possibilities. 
Choose an inclusion-minimal set $Z\subseteq\calA\setminus\calA^*$ such that $N(\calA)\setminus N^*\subseteq N(Z)$.
Such a set exists because every vertex of $N(\calA)\setminus N^*$ has a neighbor in $\calA\setminus\calA^*$. 
By minimality, every member $z\in Z$ has a vertex in $N(\calA)\setminus N^*$ adjacent to $z$ and to no other member of $Z$; this vertex is called a \emph{private neighbor} of $z$. 
These private neighbors are distinct, and hence $|Z|\le|N(\calA)\setminus N^*|\le g-|R|=h$. 
Moreover, $Z\subseteq Q\setminus\calA^*$, a set of size $m$. Once $N(\calA)\cap N^*$ and $Z$ are fixed,
\begin{equation*}
N(\calA)=(N(\calA)\cap N^*)\cup(N(Z)\setminus N^*),
\end{equation*}
so $[\calA]$ is fixed. 
The remaining members of $\calA$ can be chosen in at most $2^{a-|Z|}$ ways. 
Thus the number assigned to the good pair is at most
\begin{equation}\label{eq:reconstruct}
2^h\sum_{j=0}^{h}\binom mj2^{a-j}\leq 2^{a+h}\sum_{j=0}^{h}\binom mj.
\end{equation}

From \eqref{eq:iso-large} and Bernoulli's inequality, we obtain
\begin{equation}\label{eq:component-delta}
\frac ga\ge 1+\frac{8q}{k+q-9},
\qquad
1-\frac ag\ge \frac{8q}{k+9q-9}\ge \frac{7q}{k}
\end{equation}
for all sufficiently large $k$. The good pair inequality gives
\begin{equation}\label{eq:h-bound}
\frac hg\le \frac{\frac{q}{k}-\frac{\xi}{k}+\theta}{1+\frac{q}{k}}.
\end{equation}
Since $\theta=o(q/k)$ by Claim~\ref{claim:localcontainer}, we may suppose that $\theta\le q/(100k)$. 
Equations \eqref{eq:xi-upper} and \eqref{eq:h-bound}, together with \eqref{eq:component-delta}, then give
\begin{equation*}
\frac{k}{q}\left(1-\frac ag\right)\ge 7,
\qquad
-0.04\le \frac{\xi}{q}\le 1.01,
\qquad
0\le \frac{kh}{qg}\le 1.05.
\end{equation*}
Moreover,
\begin{equation*}
\frac mg
=1-\frac ag+\frac{\xi}{k}
=\frac qk\left[\frac{k}{q}\left(1-\frac ag\right)+\frac{\xi}{q}\right]
\ge \frac{6.96q}{k}.
\end{equation*}
The preceding bounds also imply that $h/g\le 1.05q/k$, and hence
\begin{equation*}
\frac hm\le \frac{1.05}{6.96}<\frac12.
\end{equation*}

If $h=0$, by \eqref{eq:component-delta}, then 
$$\frac1g\log \left(2^h\sum_{j=0}^{h}\binom mj2^{a-j}\right)\leq \frac1g\log \left(2^a\right)=\frac{a}{g}<1.$$

If $h\geq 1$, using $\sum_{j=0}^{h}\binom mj\le (em/h)^h$ in \eqref{eq:reconstruct}, an easy observation shows that 
\begin{align}\label{eq:entropy-calc}
\frac1g\log \left(2^h\sum_{j=0}^{h}\binom mj2^{a-j}\right)
&\le \frac1g\log \left(2^{a+h} \left(\frac{em}{h}\right)^h\right)\nonumber\\
&\le
1-\frac qk\left\{
\frac{k}{q}\left(1-\frac ag\right)
-\frac{kh}{qg}\left[
1+\log \left(
\frac{e\left(\frac{k}{q}\left(1-\frac ag\right)+\frac{\xi}{q}\right)}
{\frac{kh}{qg}}
\right)
\right]
\right\}.
\end{align}
We claim that the expression in braces is at least $1$. Keeping $\xi/q$ and $kh/(qg)$ fixed, this expression is increasing as a function of $\frac{k}{q}(1-a/g)$, since its derivative is
\begin{equation*}
1-
\frac{\frac{kh}{qg}}
{\left(\frac{k}{q}\left(1-\frac ag\right)+\frac{\xi}{q}\right)\ln 2}
\ge
1-\frac{1.05}{6.96\ln 2}>0.
\end{equation*}
It is therefore enough to replace $\frac{k}{q}(1-a/g)$ by $7$. For $h>0$, the resulting subtracted term is increasing in both $\xi/q$ and $kh/(qg)$. Indeed, its derivatives with respect to these two quantities are
\begin{equation*}
\frac{\frac{kh}{qg}}{\left(7+\frac{\xi}{q}\right)\ln 2}>0
\end{equation*}
and
\begin{equation*}
1+\log \left(
\frac{e\left(7+\frac{\xi}{q}\right)}
{\frac{kh}{qg}}
\right)-\frac1{\ln 2}
\ge
1+\log \left(\frac{6.96e}{1.05}\right)-\frac1{\ln 2}>0,
\end{equation*}
respectively. The case $h=0$ follows by continuity. Consequently,
\begin{equation*}
\frac{kh}{qg}\left[
1+\log \left(
\frac{e\left(7+\frac{\xi}{q}\right)}
{\frac{kh}{qg}}
\right)
\right]
\le
1.05\left[1+\log \left(\frac{8.01e}{1.05}\right)\right]<6.
\end{equation*}
Thus the expression in braces in \eqref{eq:entropy-calc} is greater than $7-6=1$. 
It follows that the right-hand side of \eqref{eq:entropy-calc} is at most $1-q/k$. 

After summing over the good pairs, this case contributes at most
\begin{equation*}
2^{(1-q/k+\theta)g}
\le
2^{g\left(1-\frac{3}{k\sqrt{d_X}}\right)}
\end{equation*}
sets for all sufficiently large $k$.

A short case analysis shows that the two cases together contribute at most $2^{1+g\left(1-\frac{3}{k\sqrt{d_X}}\right)}$.
Since $g>d_X^3$ implies $g/(k\sqrt{d_X})>1$, we obtain
\begin{equation*}
2^{1+g\left(1-\frac{3}{k\sqrt{d_X}}\right)}
\le
2^{g\left(1-\frac{2}{k\sqrt{d_X}}\right)}
\le
2^{g\left(1-\frac{1}{k\sqrt{d_X}}\right)},
\end{equation*}
as required.
\end{proof}

\begin{claim}\label{claim:smallclosure}
For the fixed $x$, the number of independent sets $\calA\cup\calB$ in $H$ with $\calA\ne\varnothing$ and $|[\calA]|\le T_0$ is at most
\begin{equation*}
2^{\binom{n-1}{k-1}-\binom{n-k-1}{k-1}/10}
\end{equation*}
for sufficiently large $k$.
\end{claim}

\begin{proof}
Let $\calA_1,\ldots,\calA_p$ be the $2$-linked components of $\calA$. 
We first show that their neighborhoods are pairwise disjoint. 
Suppose, for the sake of contradiction, that $N(\calA_i)\cap N(\calA_j)\neq\varnothing$ for some $i\neq j$. 
Choose $B\in N(\calA_i)\cap N(\calA_j)$, and then choose $A_i\in\calA_i$ and $A_j\in\calA_j$ adjacent to $B$. 
The vertices $A_i$ and $A_j$ have distance two in $H$, so they belong to the same $2$-linked component, a contradiction. 
Hence $N(\calA_i)\cap N(\calA_j)=\varnothing$ whenever $i\neq j$, and therefore $|N(\calA)|=\sum_{i=1}^{p}|N(\calA_i)|$.

We write $a_i=|[\calA_i]|$ and $g_i=|N(\calA_i)|$. Since $N(\calA_i)\subseteq N(\calA)$, every $C\in[\calA_i]$ satisfies $N(C)\subseteq N(\calA)$, and hence $[\calA_i]\subseteq[\calA]$. Thus $a_i\leq T_0$. Moreover, $N([\calA_i])=N(\calA_i)$, so \eqref{eq:iso-large} applied to $[\calA_i]$ gives $a_i<g_i$. Since each component is nonempty, we also have $g_i\geq d_X$.

Let
$$g_s=\sum_{i:g_i\leq d_X^3}g_i
\qquad\text{and}\qquad
g_\ell=\sum_{i:g_i>d_X^3}g_i.$$
Fix an ordering of $X$. We order the small components by their first
vertices in this ordering, and order the large components independently
in the same way. Thus every $\calA$ determines two ordered lists of
components. These two ordered lists determine the full collection of
components, so no additional interleaving information is needed.

Now fix the corresponding ordered lists of component parameters
$(a_i,g_i)$. For every $i$, we have
$\calA_i\in\calG(a_i,g_i)$. Hence the number of admissible ordered
tuples of components is at most
\begin{equation*}
\prod_{i=1}^{p}|\calG(a_i,g_i)|.
\end{equation*}
Indeed, an arbitrary tuple in this Cartesian product need not form the
collection of $2$-linked components of a common set: distinct components
must be disjoint, must have distance greater than two in $H$, and must
have disjoint neighborhoods. Ignoring these compatibility conditions
can only enlarge the count. Therefore, Claim~\ref{claim:component} gives
\begin{align*}
\prod_{i=1}^{p}|\calG(a_i,g_i)|
&\leq
\prod_{i:g_i\leq d_X^3}2^{g_i/2}
\prod_{i:g_i>d_X^3}
2^{g_i(1-1/(k\sqrt{d_X}))}\\
&=
2^{g_s/2+g_\ell(1-1/(k\sqrt{d_X}))}.
\end{align*}

For every admissible tuple, the neighborhoods of its components are
pairwise disjoint, and hence
\begin{equation*}
|N(\calA)|
=\sum_{i=1}^{p}|N(\calA_i)|
=\sum_{i=1}^{p}g_i
=g_s+g_\ell.
\end{equation*}
Consequently, \eqref{eq:Bchoices} yields at most
$2^{|Y|-g_s-g_\ell}$ choices for $\calB$. Thus, for fixed ordered
component-parameter lists, the total number of choices for
$(\calA,\calB)$ is at most
\begin{equation*}
2^{|Y|-g_s/2-g_\ell/(k\sqrt{d_X})}.
\end{equation*}

If $g_s=0$, there is one possible empty list of small-component parameters. Suppose that $g_s>0$ and that there are $p_s$ small components. Since each $g_i\geq d_X$, we have $p_s\leq g_s/d_X$. The ordered positive integers $g_i$ can be chosen in $\binom{g_s-1}{p_s-1}$ ways, and, since $a_i<g_i\leq d_X^3$, there are at most $d_X^{3p_s}$ choices for the corresponding closure sizes. Hence the number of small-component parameter lists is at most
$$\sum_{1\leq p_s\leq g_s/d_X}\binom{g_s-1}{p_s-1}d_X^{3p_s}
\leq
2^{K g_s\ln d_X/d_X}$$
for an absolute constant $K$. 
Since $\binom{g_s-1}{p_s}d_X^{3(p_s+1)}/\binom{g_s-1}{p_s-1}d_X^{3p_s}=(g_s-p_s)d^3_X/p_s>1$, the summands $\binom{g_s-1}{p_s-1}d_X^{3p_s}$ increase with $p_s$ in this range. If $m=\lfloor g_s/d_X\rfloor$, then the sum is at most
$$m\binom{g_s}{m}d_X^{3m}
\leq m(2ed_X^4)^m,$$
which gives the stated bound.

If $g_\ell=0$, there is one possible empty list of large-component parameters. Otherwise, let $p_\ell$ be the number of large components. Since every such component satisfies $g_i>d_X^3$, we have $p_\ell<g_\ell/d_X^3$. The ordered neighborhood sizes can be chosen in at most $\binom{g_\ell-1}{p_\ell-1}$ ways, while each closure size has at most $|X|$ possible values. Therefore, using $|X|\leq2^n$, the number of large-component parameter lists is at most
$$\sum_{1\leq p_\ell<g_\ell/d_X^3}\binom{g_\ell-1}{p_\ell-1}|X|^{p_\ell}
\leq
2^{K g_\ell(n+\ln d_X)/d_X^3}.$$
The summands again increase with $p_\ell$ in the stated range, and the bound follows from $\binom{m}{r}\leq(em/r)^r$.

Consequently, the total number of families under consideration is at most
\begin{equation*}
2^{|Y|}
\sum_{\substack{g_s,g_\ell\geq0\\1\le g_s+g_\ell\le |Y|}}
2^{-g_s/2-g_\ell/(k\sqrt{d_X})
+K g_s\ln d_X/d_X
+K g_\ell(n+\ln d_X)/d_X^3}.
\end{equation*}
For sufficiently large $k$, we have
$$\frac{K\ln d_X}{d_X}\leq\frac16$$
and
$$\frac{K(n+\ln d_X)}{d_X^3}\leq\frac{1}{2k\sqrt{d_X}}.$$
The second inequality follows from $n<3k$ and $d_X\geq\binom{k+2}{3}$. 
Hence every summand is at most
$$2^{-g_s/3-g_\ell/(2k\sqrt{d_X})}.$$

If $g_s>0$, then $g_s\geq d_X$. If $g_\ell>0$, then $g_\ell>d_X^3$. Since $\calA\neq\varnothing$, at least one of these cases occurs. Thus every summand has an exponential saving of at least
$$\min\left\{\frac{d_X}{3},\frac{d_X^{5/2}}{2k}\right\}
\geq\frac{d_X}{3},$$
where the last inequality follows from $d_X\geq\binom{k+2}{3}$ for all sufficiently large $k$.

Finally, there are at most $(|Y|+1)^2=2^{O(n)}$ possible pairs $(g_s,g_\ell)$. 
It follows that the total number of families is at most $2^{|Y|-d_X/3+O(n)}$.
Since $d_X\geq\binom{k+2}{3}\gg n$, the term $O(n)$ can be absorbed into the saving. 
Note that $d_X=\binom{n-k-1}{k-1}$.
Recalling that $|Y|=\binom{n-1}{k-1}$ and weakening the saving from $d_X/3$ to $d_X/10$, 
we obtain $$2^{\binom{n-1}{k-1}-\binom{n-k-1}{k-1}/10},$$ as required.
\end{proof}

\begin{proof}[Proof of Theorem~\ref{thm:nearstar}]
A non-trivial family in this class has $\calA\ne\varnothing$. 
If $|[\calA]|>T_0$, Claim~\ref{claim:largeclosure} gives at most $2^{(1-\gamma/2)\binom{n-1}{k-1}}$ possibilities. 
Otherwise, Claim~\ref{claim:smallclosure} gives at most $2^{\binom{n-1}{k-1}-\binom{n-k-1}{k-1}/10}$. 
Therefore, the number of non-trivial intersecting families $\calF$ satisfying 
$|\calF\setminus\calS_x|\le\zeta\binom{n-1}{k-1}$ is at most $2^{(1-\gamma/2)\binom{n-1}{k-1}}+2^{\binom{n-1}{k-1}-\binom{n-k-1}{k-1}/10}$.
\end{proof}

\section{Concluding Remarks}

The remaining case is $n=2k+2$, and it would be interesting to construct containers indexed by a nearby full star. 
In the present proof, we first construct global containers in the Kneser graph and then apply a stability result to show that every sufficiently large container is close to some full star $\calS_x$. 
A different approach would be to choose the center $x$ during the container construction and describe each intersecting family relative to $\calS_x$.

This may be useful for the boundary case. Indeed, when $q=n-2k$ is fixed, the present estimates give
$$
\log|\mathfrak C|=O\left(\frac{\binom{n-1}{k-1}\ln k}{k^{q-1}}\right),
$$
whereas the natural saving in the exponent near a full star has order
$
\binom{n-1}{k-1}/k.
$
The ratio between these two quantities has order $\ln k/k^{q-2}$, which tends to zero when $q\geq3$ but not when $q=2$. 
Thus, a direct extension of the present global container argument does not appear to cover $n=2k+2$. 
A container construction indexed from the beginning by a nearby full star may avoid this extra logarithmic factor and provide a possible way to study the remaining case.

\section*{Acknowledgements}
\noindent This research is supported by National Key R\&D Program of China under grant number 2024YFA1013900, NSFC under grant number 12471327, 
and the China Postdoctoral Science Foundation under Grant Number 2026M793375.
\section*{Declaration}
	
\noindent$\textbf{Conflict~of~interest}$
The author declares that they have no known competing financial interests or personal relationships that could have appeared to influence the work reported in this paper.
\vskip 2mm	
\noindent$\textbf{Data~availability}$
No data was used for the research described in the article.

\end{document}